\documentclass[11pt]{amsart}
\usepackage{amssymb}
\usepackage[margin=1in]{geometry}
\usepackage[colorlinks]{hyperref}
\newtheorem{theorem}{Theorem}[section]

\newtheorem{prop}[theorem]{Proposition}
\newtheorem{cor}[theorem]{Corollary}

\theoremstyle{definition}

\newtheorem{example}[theorem]{Example}

\theoremstyle{remark}
\newtheorem{remark}[theorem]{Remark}

\numberwithin{equation}{section}
\begin{document}

\newcommand{\spacing}[1]{\renewcommand{\baselinestretch}{#1}\large\normalsize}
\spacing{1.14}

\title{Conformally related invariant $(\alpha,\beta)$-metrics on homogeneous spaces}

\author {Azar Fatahi}

\address{Azar Fatahi\\ Department of Pure Mathematics \\ Faculty of  Mathematics and Statistics\\ University of Isfahan\\ Isfahan\\ 81746-73441-Iran.} \email{}

\author {Masoumeh Hosseini}

\address{Masoumeh Hosseini\\ Department of Pure Mathematics \\ Faculty of  Mathematics and Statistics\\ University of Isfahan\\ Isfahan\\ 81746-73441-Iran.} \email{hoseini\_masomeh@ymail.com}

\author {Hamid Reza Salimi Moghaddam}

\address{Hamid Reza Salimi Moghaddam\\ Department of Pure Mathematics \\ Faculty of  Mathematics and Statistics\\ University of Isfahan\\ Isfahan\\ 81746-73441-Iran.\\ Scopus Author ID: 26534920800 \\ ORCID Id:0000-0001-6112-4259\\} \email{hr.salimi@sci.ui.ac.ir and salimi.moghaddam@gmail.com}


\date{\today}

\begin{abstract}
In this paper, we give the flag curvature formula of general $(\alpha,\beta)$-metrics of Berwald type. We study conformally related $(\alpha,\beta)$-metrics, especially general $(\alpha,\beta)$-metrics that are conformally related to invariant $(\alpha,\beta)$-metrics. Also, a necessary and sufficient condition for a Finsler metric conformally related to an $(\alpha,\beta)$-metric is given and conformally related Douglas Randers metrics are studied. Finally, we present some examples of conformally related $(\alpha,\beta)$-metrics.
\end{abstract}

\maketitle

\section{\textbf{Introduction}}
In the last two decades, many mathematicians worked on invariant Finsler metrics on homogeneous spaces (see \cite{Deng} and \cite{Deng-Hosseini-Liu-Salimi}). Among different types of Finsler metrics, $(\alpha,\beta)$-metrics have been paid more attention to because of their simplicity and applications in physics (see \cite{Antonelli-Ingarden-Matsumoto}, \cite{Asanov} and \cite{Chern-Shen}). These Finsler metrics were introduced in \cite{Matsumoto}, by Matsumoto. In fact, the Randers metric, the first $(\alpha,\beta)$-metric defined by G. Randers in 1941, was introduced because of its application in general relativity (see \cite{Randers}). Other examples of such $(\alpha,\beta)$-metrics defined because of their applications in physics are the Matsumoto metric and the Kropina metric (see \cite{Antonelli-Ingarden-Matsumoto}, \cite{Asanov} and \cite{Yoshikawa-Sabau}).\\
Suppose that $g$ is a Riemannian metric and $\beta$ is a 1-form on a differentiable manifold $M$.
For a $C^\infty$ function $\phi : (-b_0 , b_0) \longrightarrow \mathbb{R^+}$ satisfying
\begin{equation}
\phi (s) - s \phi ^{'} (s) + ( b^2 - s^2 ) \phi ^{''} (s) > 0, \qquad  \vert s \vert  \leq b < b_0,
\end{equation}
and  $\|\beta\|_\alpha < b_0$ (see  \cite{Chern-Shen}), the Finsler metric $F = \alpha \phi (\frac{\beta}{\alpha})$ is called an $(\alpha , \beta)$-metric, where $\alpha(x,y)=\sqrt{g(y,y)}$.
To define the above-mentioned $(\alpha,\beta)$-metrics, Randers, Matsumoto, and Kropina metrics, it is sufficient to consider the function $\phi$ as $\phi(s)=1+s$, $\phi(s)=\frac{1}{1-s}$, and $\phi(s)=\frac{1}{s}$, respectively.\\

In the year 2011, Yu and Zhu generalized the concept of the $(\alpha,\beta)$-metric to a more general case which was called the general $(\alpha,\beta)$-metric (see \cite{Yu-Zhu}). A Finsler metric $F$ on a differentiable manifold $M$ is called a general $(\alpha,\beta)$-metric if there exists a $C^\infty$ function $\phi$, a Riemannian metric $g$, and a 1-form $\beta$, such that
\begin{equation}\label{general alpha beta metric}
    F=\alpha\phi(x,\frac{\beta}{\alpha}),
\end{equation}
where $x\in M$ and $\alpha(x,y)=\sqrt{g(y,y)}$. \\
A special class of general $(\alpha,\beta)$-metrics is of the form $F=\alpha\phi(b^2,\frac{\beta}{\alpha})$, where $b^2:=\|\beta\|_\alpha^2$, $|s|\leq b<b_0$, for some $0<b_0\leq +\infty$. This family of general $(\alpha,\beta)$-metrics is important because it includes some Bryant Finsler metrics (see \cite{Yu-Zhu}). It is shown that for $\|\beta\|_\alpha<b_0$ the function $F=\alpha\phi(b^2,\frac{\beta}{\alpha})$ is a Finsler metric if and only if $\phi$ is a positive $C^\infty$ function such that
\begin{itemize}
  \item $\phi-s\phi_2>0$, \ \ \ $\phi(s)-s\phi_2(s) + (b^2-s^2)\phi_{22}(s) > 0$, (if $\dim M\geq3$)
  \item $\phi(s)-s\phi_2(s) + (b^2-s^2)\phi_{22}(s) > 0$, (if $\dim M=2$),
\end{itemize}
where $|s|\leq b<b_0$ (see \cite{Yu-Zhu}).\\
It seems that the study of geometric properties of invariant general $(\alpha,\beta)$-metrics on homogeneous spaces is interesting. But, we see this is not a good idea, because every invariant general $(\alpha,\beta)$-metric, which is defined by an invariant Riemannian metric and an invariant vector field (1-form), is an invariant $(\alpha,\beta)$-metric (see proposition 3.1 below). So we study a family of general $(\alpha,\beta)$-metrics which are very close to invariant $(\alpha,\beta)$-metrics. In this work, we study general $(\alpha,\beta)$-metrics that are conformally related to invariant $(\alpha,\beta)$-metrics defined by an invariant Riemannian metric and an invariant vector field. Also we give the flag curvature formula of general $(\alpha,\beta)$-metrics of Berwald type. A necessary and sufficient condition for a Finsler metric conformally related to an $(\alpha,\beta)$-metric is given. Conformally related Douglas Randers metrics are investigated. Finally, some examples of conformally related $(\alpha,\beta)$-metrics are given.


\section{\textbf{Conformally related $(\alpha,\beta)$-metrics}}
In this section, firstly we consider the general $(\alpha,\beta)$-metrics of Berwald type. Easily, similar to the $(\alpha,\beta)$-metrics, we compute the flag curvature formula of the general $(\alpha,\beta)$-metrics of Berwald type. Next, we turn our attention to the Finsler metrics that are conformally related to the $(\alpha,\beta)$-metrics. We give a necessary and sufficient condition for such metrics to be of Berwald type. Also, we study such Randers metrics that are of Douglas type.

\begin{remark}\label{24}
Let $F(x,y)=\alpha\phi(b^2,\frac{\beta}{\alpha})$ be a general $(\alpha,\beta)$-metric on a differentiable manifold $M$. If $\beta$ is parallel with respect to $\alpha$, then $F$ is of Berwald type.
\end{remark}
\begin{proof}
Assume that $\beta$ is parallel with respect to $\alpha$ i.e. $b_{i;j} =0$, where $b_{i;j}$ is the covariant derivative of $b_i$ with respect to the Riemannian metric $g$. Suppose that $G^i$ and $G^{i}_\alpha$ denote the spray coefficients of $F$ and $\alpha$, respectively. Now, Proposition 3.4 of \cite{Yu-Zhu} shows that  $G^i=G^{i}_\alpha$ and so $F$ is of Berwald type.
\end{proof}

\begin{prop}
Let $M$ be a Finsler manifold equipped with a general $(\alpha,\beta)$-metric
$F(x,y)=\alpha \phi(b^2,\frac{\beta}{\alpha})$. If $\beta$ is parallel with respect to $\alpha$, then the flag curvature $K^F(y,P)$ of $F$ is given by
\begin{equation}
K^F(y,P)=\frac{\alpha^2 \|{u}\|^2_{\alpha} - g(y,u)^2}{F^2 g^F_y(u,u) - g^F_y(y,u)^2}\rho K^g(P).\label{E}
\end{equation}
where $P= span\{y,u\}$, $\rho = \phi(\phi-s\phi_2)$, and $K^g(P)$ is the sectional curvature of the Riemannian metric $g$.\\
In a particular case, if $\{y,u\}$ is an orthonormal set with respect to the Riemannian metric $g$, then
\begin{equation}
K^F(y,P)=\frac{1}{\phi^2(1+g^2(X,u)D)}K^g(P),
\end{equation}
where
\begin{equation*}
D:=\frac{\phi_{22}}{\phi-s\phi_2},
\end{equation*}
and $X$ is the vector field corresponding to the 1-form $\beta$ with respect to the Riemannian metric $g$.
\end{prop}
\begin{proof}
For a local coordinate system $(x^i)$ let $s=\frac{\beta}{\alpha}$, where $\alpha(x,y)=\sqrt{g_{ij}y^i y^j}$ and $\beta(x,y)=b_i y^i$. Using Proposition 3.2 of \cite{Yu-Zhu} for the Hessian matrix $g^F_{ij}$ of $F$ we have:
\begin{equation*}
g^F_{ij}=\rho g_{ij} +\rho_0 b_i b_j + \rho_1(b_i\alpha_{y^j}+b_j\alpha_{y^i})-s\rho_1 \alpha_{y^i}\alpha_{y^j}.
\end{equation*}
where
\begin{equation*}
\rho=\phi^2-s\phi\phi_2=\phi(\phi-s\phi_2) , \quad \rho_0= \phi\phi_{22}+\phi_2\phi_2 , \quad \rho_1=(\phi-s\phi_2)\phi_2-s\phi\phi_{22}.
\end{equation*}
Suppose that $\beta$ is parallel with respect to $\alpha$ i.e. $b_{i;j} =0$. According to Remark
\eqref{24}, we have $G^i=G^{i}_\alpha$. The Riemannian curvature of $F$ is given by
\begin{equation}
R^{i}_k=2\frac{\partial G^i}{\partial x^k}-y^j\frac{\partial^2 G^i}{\partial x^j \partial y^k}+2G^j\frac{\partial^2 G^i}{\partial y^j\partial y^k}- \frac{\partial G^i\partial G^j}{\partial y^j\partial y^k}.\label{A}
\end{equation}
The formula of Riemannian curvature \eqref{A} implies that
\begin{equation*}
R^{i}_j=\sideset{^\alpha}{^i_{j}}{\mathop{R}}.
\end{equation*}
where $R^{i}_j$ and $\sideset{^\alpha}{^i_{j}}{\mathop{R}}$ are the Riemannian curvatures of $F$ and $\alpha$, respectively. Now, let
\begin{equation*}
R_{ij}:=g^F_{im}R^{m}_j , \sideset{^\alpha}{_{ij}}{\mathop{R}}:=g_{im} \sideset{^\alpha}{^m_{j}}{\mathop{R}}.
\end{equation*}
Using the fact
\begin{equation*}
\alpha_{y^m} \sideset{^\alpha}{^m_{j}}{\mathop{R}}=\frac{1}{\alpha} g_{im} y^i \sideset{^\alpha}{^m_{j}}{\mathop{R}}= \frac{1}{\alpha} y^i \sideset{^\alpha}{_{ij}}{\mathop{R}}= 0,
\end{equation*}
and by a direct computation we have
\begin{align}
R_{ij} &=g^F_{im} \sideset{^\alpha}{^m_{j}}{\mathop{R}}\label{B}\\
&=(\rho g_{im} + \rho_0 b_i b_m + \rho_1 (b_i \alpha_{y^m} +b_m \alpha_{y^i}) - s\rho_1 \alpha_{y^i} \alpha_{y^m}) \sideset{^\alpha}{^m_{j}}{\mathop{R}} \nonumber\\
&=\rho \sideset{^\alpha}{_{ij}}{\mathop{R}}.\nonumber
\end{align}
Now, since $b_{i;j}=0$, the Ricci identity implies that
\begin{equation*}
b_m \sideset{^\alpha}{^m_{j}}{\mathop{R}}=b_m \sideset{^\alpha}{^m_{i jk}}{\mathop{R}} = b_{i;j;k} - b_{i;k;j} =0.
\end{equation*}
So
\begin{equation*}
 b_m \sideset{^\alpha}{^m_{j}}{\mathop{R}}=b_m \sideset{^\alpha}{^m_{ijk}}{\mathop{R}} y^i y^j =0.
\end{equation*}
By the definition, for $P=span\{u,y\}$ and $ u=u^i \frac{\partial}{\partial x^i}$, the flag curvature $K^F$ of $F$ and the sectional curvature $K^{g}$ are given by
\begin{equation}
K^F(y,P)=\frac{g^F_{y}(R_{y}(u),u)}{g^F_y(y,y)g^F_y(u,u) - (g^F_{y}(y,u))^2}=\frac{R_{ij} u^i u^j}{F^2 g^F_y(u,u) - (g^F_{y}(u,y))^2},\label{C}
\end{equation}
and
\begin{equation}
K^{g}(y,p)=\frac{g(\sideset{^\alpha}{_{y}}{\mathop{R}}(u),u)}{g(y,y) g(u,u) - g^2(y,u)} = \frac{\sideset{^\alpha}{_{ij}}{\mathop{R}} u^i u^j}{\alpha^2 g(u,u) - g^2(y,u)}.\label{D}
\end{equation}
The relations \eqref{B},\eqref{C} and \eqref{D}, imply that \eqref{E} holds. \\
It can be shown that, if $\{y,u\}$ is an orthonormal basis of $P$ with respect to the Riemannian metric $g$, then

\begin{align*}
g^F_y(u,u)&=g^F_{ij} u^i u^j =(\rho g_{ij} + \rho_0 b_i b_j +\rho_1 (b_i \alpha_{y^j} +b_j \alpha_{y^i}) - s \rho_1 \alpha_{y^i} \alpha_{y^j}) u^i u^j\\
 &=\rho + \rho_0 g^2(X,u),
\end{align*}
and
\begin{align*}
g^F_y(y,u) &= g^F_{ij} y^i u^j =( \rho g_{ij} + \rho_0 b_i b_j + \rho_1 (b_i \alpha_{y^j} + b_j \alpha _{y^i} - s \rho_1 \alpha _{y^i}  \alpha_{y^j} ) y^i u^j\\
&= \rho_0 \beta g(X,u) + \rho_1 \alpha g(X,u) = (s \rho_0 + \rho_1) \alpha g(X,u)\\
&= \phi \phi_2 g(X, u).
\end{align*}
On the other hand, we have
\begin{equation}
g_{ij} u^i u^j=g(u,u)=1,\quad  b_i u^i=g(X,u),\quad  \alpha_{y^i} u^i = \frac{1}{\alpha} g(y,u)=0.\label{F}
\end{equation}
It follows that
\begin{align}
F^2 g^F_y(u,u) - (g^F_{y}(y,u))^2 &= \phi^2 (\rho + \rho_0 g^2(X,u)) - (\phi \phi_2 g(X,u))^2\label{G}\\
&=\phi^2 \rho +g^2(X,u) (\phi^2 \rho_0 - (\phi \phi_2)^2)\nonumber\\
&=\phi^2 \rho + g^2(X,u) \phi^3 \phi_{22}\nonumber.
\end{align}
Finally, using \eqref{F}, \eqref{G} and the relation $D=\frac{\phi_{22}}{\phi-s\phi_2}$, we have
\begin{equation*}
K^F(y,P)=\frac{\rho}{\phi^2 \rho + g^2 (X,u) \phi^3 \phi_{22}} K^{g} (P) = \frac{1}{\phi^2(1+ g^2(X,u)D)} K^{g}(P).
\end{equation*}
\end{proof}

\begin{prop}
If $\tilde{F}$ is a Finsler metric conformally equivalent to an $(\alpha,\beta)$-metric $F$ then $\tilde{F}$ is an $(\alpha,\beta)$-metric.
\end{prop}
\begin{proof}
Suppose that $M$ is an arbitrary differentiable manifold equipped with an $(\alpha,\beta)$-metric $F$, defined by a Riemannian metric $g$ and a vector field $X$. Let $\tilde{F}= e^f F$ be a Finsler metric conformally equivalent to $F$. Suppose that $F$ is defined by $F=\alpha\phi(\frac{\beta}{\alpha})$ where $\phi:(b_0,b_0) \to \mathbb{R}$ is a $C^\infty$ function.  Now, we define a vector field $\tilde{X}$ and a Riemannian metric $\tilde{g}$ as follows:
\begin{equation}
\tilde{X}=e^{-f}X, \ \ \ \tilde{g}=e^{2f} g.\label{31}
\end{equation}
Let $\tilde{\phi}=\phi$, so we have
\begin{equation*}
\tilde{\alpha} \tilde{\phi}(\frac{\tilde{\beta}}{\tilde{\alpha}})(x,y)
=\sqrt{e^{2f(x)}g(y,y)}\phi(\frac{e^{2f(x)}g(e^{-f(x)}X,y)}{\sqrt{e^{2f(x)}g(y,y)}}=e^{f(x)}F(x,y)=\tilde{F}
\end{equation*}
Also we have $\|{\tilde{X}}\|_{\tilde{\alpha}}< b_0$.
\end{proof}

In the next proposition, we give a necessary and sufficient condition for a Finsler metric $\tilde{F}$ conformally equivalent to an $(\alpha,\beta)$-metric $F$, to be of Berwald type.

\begin{prop}
Let $M$ be an arbitrary differentiable manifold. Suppose that $F$, $\tilde{F}$, $g$, $\tilde{g}$, $X$ and $\tilde{X}$ are as the previous proposition. Let $\nabla$ and $\tilde{\nabla}$ be the Levi-Civita connections of the Riemannian metrics $g$ and $\tilde{g}$, respectively. Then, the Finsler metric $\tilde{F}$ is of Berwald type if and only if for any vector field $Y$ on $M$ we have
\begin{equation}
\nabla_Y X= g(X,Y)\nabla{f}-XfY \label{4}
\end{equation}\label{9}
\end{prop}
\begin{proof}
We know that the Finsler metric $\tilde{F}$ is of Berwald type if and only if the vector field $\tilde{X}$ is parallel with respect to $\tilde{g}$, that is, for any vector field $Y$ on $M$ $\tilde{\nabla}_{Y}{\tilde{X}}=0$. On the other hand, by Lemma 1 of \cite{Kuhnel}, we have
\begin{equation}
\tilde{\nabla}_{Y} \tilde{X}=\nabla_{Y} \tilde{X}+(Yf)\tilde{X}+(\tilde{X}f)Y-g(\tilde{X},Y)\nabla{f}.
\end{equation}
So $\tilde{F}$ is a Berwald metric if and only if
\begin{equation}
\nabla_{Y} \tilde{X}+(Yf)\tilde{X}+(\tilde{X}f)Y-g(\tilde{X},Y)\nabla{f}=0.
\end{equation}
Now, suppose that $\tilde{X}=e^{-f} X$, then we have,
\begin{equation}
(Y e^{-f})X+e^{-f}(\nabla_{Y} X+(Yf)X+(Xf)Y-g(X,Y)\nabla{f})=0.\label{30}
\end{equation}
A direct computation shows that equation \eqref{30} is equivalent to
\begin{equation}
\nabla_Y X= g(X,Y)\nabla{f}-XfY.
\end{equation}
\end{proof}
In \cite{Zhu}, Zhu studied general $(\alpha,\beta)$-metrics with vanishing Douglas curvature.
Also in \cite{Matveev-Saberali}, the authors studied two-dimensional conformally related Douglas
metrics and showed that such metrics are Randers.
In the following proposition, we study conformally related Douglas Randers metrics in arbitrary dimension.
\begin{prop}
Let $F=\alpha +\beta$ be a Randers metric and $\tilde{F}$ be conformally related to $F$. Then, $\tilde{F}$ is of Douglas type if and only if
\begin{equation}
 d\beta(Y,Z)+Yfg(X,Z)-Zfg(X,Y)=0\qquad \forall Y,Z\in\mathcal{X}(M). \label{3}
\end{equation}\label{10}
\end{prop}
\begin{proof}
Let $\tilde{F}=\tilde{\alpha}+\tilde{\beta}$ be a Randers metric that is $\tilde{\alpha}(x,y)=\sqrt{\tilde{g}_x(y,y)}$ and $\tilde{\beta}(x,y)=\tilde{g}(\tilde{X}(x),y)$. We know that the Finsler metric $\tilde{F}$ is of Douglas type if and only if the 1-form $\tilde{\beta}$ is closed i.e. $d\tilde{\beta}=0$. On the other hand, according to Proposition 14.29 of \cite{Lee} we have
\begin{equation}
d\tilde{\beta}(Y,Z)=Y\tilde{\beta}(Z)-Z\tilde{\beta}(Y)-\tilde{\beta}[Y,Z] \qquad \forall Y,Z
\end{equation}
It shows that $\tilde{F}$ is a Douglas metric if and only if
\begin{equation}
Y\tilde{g}(\tilde{X},Z)-Z\tilde{g}(\tilde{X},Y)-\tilde{g}(\tilde{X},[Y,Z])=0.
\end{equation}
We replace $\tilde{X}$ and $\tilde{g}$ with $e^{-f}X$ and $e^{2f}g$ respectively in the above equality obtaining
\begin{equation}
Y(e^fg(X,Z))-Z(e^fg(X,Y))-e^fg(X,[Y,Z])=0.\label{32}
\end{equation}
By direct computation, \eqref{32} is equivalent to
\begin{equation}
 d\beta(Y,Z)+Yfg(X,Z)-Zfg(X,Y)=0.
\end{equation}
\end{proof}
\begin{cor}
Let $F=\alpha +\beta$ be a Randers metric of Douglas type  and $\tilde{F}$ be conformally related to $F$, $\tilde{F}=e^f F$, then, $\tilde{F}$ is of Douglas type if and only if $f_{i}(x)=\frac{\partial f}{\partial x_i}(x)$ proportional to $ b_{i}(x)$ i.e. $f_{i} b_{j}-f_{j}b_{i}=0$, where $\beta=b_{i}(x) y^i$.
\end{cor}
\begin{proof}
Since $F$ is a Randers metric of Douglas type hence, $d \beta =0$. Suppose $Y=\frac{\partial}{\partial x_i}$ and $Z=\frac{\partial}{\partial x_j}$. According to Proposition \eqref{10} we have $\frac{\partial f}{\partial x_i} b_{j} -\frac{\partial f}{\partial x_j} b_{i}=0$.
\end{proof}


\section{\textbf{Conformally related invariant $(\alpha,\beta)$-metrics}}
In this short section, we study conformally related invariant $(\alpha,\beta)$-metrics on homogeneous spaces. In the following proposition, easily we see that there is no nontrivial $G$-invariant general $(\alpha,\beta)$-metric on a homogeneous space $G/H$.

\begin{prop}
Let $F=\alpha \phi(x,\frac{\beta}{\alpha})$ be an $(\alpha,\beta)$-metric which is  defined by a $G$-invariant vector field $X$ and a $G$-invariant Riemannian metric $g$ on $M=G/H$. Then, $F$ is a $G$-invariant general $(\alpha,\beta)$-metric if and only if $F$ is a $G$-invariant $(\alpha,\beta)$-metric. In the special case any left-invariant general $(\alpha,\beta)$-metric defined by a left-invariant vector field and a left-invariant Riemannian metric is a left-invariant $(\alpha,\beta)$-metric.
\end{prop}
\begin{proof}
Let  $\tau_{a}:G/H\to G/H$ be a diffeomorphism that $\tau_{a}(xH)=axH\qquad \forall a, x\in G$. The Riemannian metric $g$ and the vector field $X$ are $G$-invariant that is
\begin{align}
&g_{xH}(d\tau_{x}Y,d\tau_{x}Z)=g_{eH}(Y,Z) \label{X}\\
&d\tau_{x}X=X.\label{Z}
\end{align}
By \eqref{X} and \eqref{Z}  $F$ is an $(\alpha,\beta)$-metric.
\end{proof}

\begin{remark}
Let $g$ be a $G$-invariant Riemannian metric and $X$ be a  $G$-invariant vector field on the homogeneous space  $M=G/H$. Suppose that $f:M\to \mathbb{R}$ is a smooth function such that $f(H)=0$ and the vector field $\tilde{X}$ and the Riemannian metric $\tilde{g}$ are defined as follows
\begin{equation}
\tilde{X}=e^{-f}X ,\tilde{g}=e^{2f} g.\label{1}
\end{equation}
Clearly $\tilde{X}$ and $\tilde{g}$ are not necessarily $G$-invariant but the two Riemannian metrics $g$ and $\tilde{g}$ on $M$  are conformally related. Suppose that
$\tilde{F}=\tilde{\alpha} \tilde{ \phi}(\tilde{b}^2,\frac{\tilde{\beta}}{\tilde{\alpha}})$ is a general $(\alpha,\beta)$-metric on $G/H$, where $\tilde{\alpha}$ is the norm of the metric $\tilde{g}$, $\tilde{\beta}$ is the 1-form defined by $\tilde{X}$, and  $\tilde{b}^2= \tilde{g}(\tilde{X},\tilde{X})$. Now, we assume
$\phi:(-b_0,b_0) \to \mathbb{R}$ such that $\phi(s):=\tilde{\phi}(\tilde{b}^2(H),s)$. Easily, $\phi$ is a $C^{\infty}$ function. It can be shown that $F=\alpha\phi(\frac{\beta}{\alpha})$ is a $G$-invariant $(\alpha,\beta)$-metric on $M=G/H$, where $X=\tilde{X}(H)$ and $\beta(y)=g(X,y)$. Furthermore $\tilde{F}$ is conformally related to $F$.
\end{remark}

We now turn to the left-invariant metrics on the Lie groups.
\begin{prop}
Suppose that $F=\alpha+\beta$ is a left-invariant Randers metric defined by a left-invariant vector field $X$ and a left-invariant Riemannian metric $g$ on a Lie group $G$. Let $\tilde{F}=e^{f} F$. If $\tilde{F}$ is of Douglas type then for all left-invariant vector fields $Y, Z$ we have
\begin{equation}
g(X,[Z,Y])+Yf g(X,Z)-Zf g(X,Y)=0. \label{32}
\end{equation}
\end{prop}
\begin{proof}
According  to Proposition 14.29 of \cite{Lee} we have
\begin{equation}
d\beta(Y,Z)=Yg(X,Z)-Zg(X,Y)-g([Y,Z],X).
\end{equation}
Since $g$ is a left-invariant Riemannian metric so
\begin{equation}
d\beta(Y,Z)=g(X,[Y,Z])\label{8}
\end{equation}
Due to \eqref{3} and \eqref{8} we conclude that \eqref{32} holds and it completes the proof.
\end{proof}

\section{\textbf{Examples}}
In this section, using the results obtained in the previous sections, we give some examples of $(\alpha,\beta)$-metrics that are conformally related to Randers metric under which conditions they are of Douglas type. Also for a certain $X$ and $f$, we show that $\tilde{F}$, which is the conformally related to an $(\alpha,\beta)$-metric, is of Berwald type.

\subsection{The Heisenberg group $H_3$}

The Heisenberg group $H_3$ can be considered as the Euclidean space $\mathbb{R}^3$ with the following multiplication
\begin{equation}
(x',y',z').(x,y,z)=(x'+x,y'+y,z'+z+\frac{1}{2}yx'-\frac{1}{2}y'x).
\end{equation}
Let $g$ be the left-invariant Riemannian metric on $H_3$ such that the left-invariant basis
\begin{equation}
 \{e_1=\frac{\partial}{\partial x}-\frac{y}{2}\frac{\partial}{\partial z}, e_2=\frac{\partial}{\partial y}+\frac{x}{2}\frac{\partial}{\partial z},e_3=\frac{\partial}{\partial z} \},
\end{equation}
is an orthonormal basis.
Easily, we can see
\begin{equation}
[e_1,e_2]=e_3, [e_1,e_3]=[e_2,e_3]=0.
\end{equation}

\begin{prop}
Suppose that $G=H_3$ is the Heisenberg Lie group, and $F=\alpha+\beta$ is a Randers metric defined by the left-invariant Riemannian metric $g$ (which is defined above) and a left-invariant vector field on $H_3$. Let $\tilde{F}=e^f F$ be conformally related to $F$. Then $\tilde{F}$ is a Douglas metric if and only if
\begin{align}
&X=ae_1+be_2, \qquad a,b\in \mathbb{R}\\
&\frac{\partial f}{\partial z}=0,\label{18}\\
&b\frac{\partial f}{\partial x}-a\frac{\partial f}{\partial y}=0.\label{19}
\end{align}
\end{prop}

\begin{proof}
If $X=ae_1+be_2$ and $f$ satisfy the conditions \eqref{18} and \eqref{19}, easily it can be seen the relation \eqref{3} holds, and it shows that $\tilde{F}$ is of Douglas type.\\
Conversely, if $X=ae_1+be_2+ce_3$ and $\tilde{F}$ is a Douglas metric, based on Proposition \eqref{10}, we have
\begin{equation*}
 d\beta(Y,Z)+Yfg(X,Z)-Zfg(X,Y)=0,
\end{equation*}
which is satisfied for all left-invariant vector fields $Y$ and $Z$.\\
In the special case where $Y=e_i$ and $Z=e_j$ $(i<j, i,j\in\{1,2,3\})$, we have
\begin{align}
&b\frac{\partial f}{\partial x}-a\frac{\partial f}{\partial y}-(by+a\frac{x}{2})\frac{\partial f}{\partial z}=c,\label{11}\\
&c\frac{\partial f}{\partial x}-(cy+a)\frac{\partial f}{\partial z}=0,\label{12}\\
&c\frac{\partial f}{\partial y}+(c\frac{x}{2}-b)\frac{\partial f}{\partial z}=0.\label{13}
\end{align}
Then, by \eqref{11}, \eqref{12} and \eqref{13}, we have $c=0$, $\frac{\partial f}{\partial z}=0$ and $b\frac{\partial f}{\partial x}-a\frac{\partial f}{\partial y}=0$.
\end{proof}

\begin{cor}
\begin{description}
  \item[a] If $X=be_2$ then, $\tilde{F}$ is of Douglas type if and only if $f=f(y)$.
  \item[b] If $X=ae_1$ then,$\tilde{F}$ is of Douglas type if and only if $f=f(x)$.
\end{description}
\end{cor}

\subsection{The Lie group $\Bbb{R}\rtimes\Bbb{R}^{+}$}

Let  $G$ be the two-dimensional solvable Lie group $\Bbb{R}\rtimes\Bbb{R}^{+}$ and $g$ be the left-invariant Riemannian metric on  $G=\Bbb{R}\rtimes\Bbb{R}^{+}$ such that the left-invariant basis
\begin{equation}
\{e_1=y\frac{\partial}{\partial y},e_2=y\frac{\partial}{\partial x}\},
\end{equation}
is an orthonormal basis.
Easily, we can see
\begin{equation}
[e_1,e_2]=e_2.
\end{equation}
Suppose that $F=\alpha+\beta$ is a Randers metric defined by the above left-invariant Riemannian metric $g$ and a left-invariant vector field $X=ae_1+be_2 (a,b\in \mathbb{R})$.
\begin{prop}
Suppose that $G$ is the Lie $\Bbb{R}\rtimes\Bbb{R}^{+}$, and $F=\alpha+\beta$ is a Randers metric as above. Let $\tilde{F}=e^f F$ be conformally related to $F$, then $\tilde{F}$ is a Douglas metric  if and only if
\begin{align}
&by\frac{\partial f}{\partial y}-ay\frac{\partial f}{\partial x}=b.\label{20}
\end{align}
\end{prop}
\begin{proof}
If $X=ae_1+be_2$ is a left-invariant vector field, and $f$ satisfies the conditions \eqref{20}, easily it can be seen the relation \eqref{3} holds and it is shown that $\tilde{F}$ is of Douglas type.\\
Conversely, if $\tilde{F}$ is a Douglas metric and $X=ae_1+be_2$, based on Proposition \eqref{10}, we have
\begin{equation}
 d\beta(Y,Z)+Yfg(X,Z)-Zfg(X,Y)=0\nonumber
\end{equation}
which is satisfied for all left-invariant vector fields $Y$ and $Z$.\\
In the special case where $Y=e_1$ and $Z=e_2$ we have
\begin{equation}
by\frac{\partial f}{\partial y}-ay\frac{\partial f}{\partial x}=b.
\end{equation}

\end{proof}
\begin{cor}
\begin{description}
  \item[a] If $X=be_2$ then, $\tilde{F}$ is of Douglas type if and only if $f=lny+g(x)$.
  \item[b] If $X=ae_1$ then,$\tilde{F}$ is of Douglas type if and only if $f=f(y)$.
\end{description}
\end{cor}

\begin{example}
Let $G=\Bbb{R}\rtimes\Bbb{R}^{+}$ and $F$ be an $(\alpha,\beta)$-metric defined by the above left-invariant  Riemannian metric $g$ and a left-invariant vector field $X$. For the Levi-Civita connection of $g$, we can see that $\nabla_{e_2}{e_1}=-e_2$ and $\nabla_{e_1}{e_1}=0$. Let $X=ae_1 (a\in\mathbb{R})$, $f(x,y)=lny$ and $Y=y_1e_1+y_2e_2 (y_1,y_2\in C^{\infty}(G))$. Easily we have $\nabla_{Y}{X}=-ay_2y\frac{\partial}{\partial x}$, $g(X,Y)=ay_1$, $Xf=a$ and $\nabla f=y\frac{\partial}{\partial y}$. Therefore, according to the proposition\eqref{9}, $\tilde{F}$ is of Berwald type.
\end{example}

\subsection{The Lie group $\Bbb{R}^2\rtimes\Bbb{R}^{+}$}
Let  $G=\Bbb{R}^2\rtimes\Bbb{R}^{+}$ and $g$ be the left-invariant Riemannian metric on $G$ such that the left-invariant basis
\begin{equation}
 \{e_1=z\frac{\partial}{\partial z}, e_2=z\frac{\partial}{\partial x}, e_3=z\frac{\partial}{\partial y}\},
\end{equation}
is an orthonormal basis.
Clearly, we can see
\begin{equation}
 [e_1,e_2]=e_2, [e_1,e_3]=e_3,  [e_2,e_3]=0.
\end{equation}
Suppose that $F=\alpha+\beta$ is a Randers metric defined by $g$ and a left-invariant vector field $X=ae_1+be_2+ce_3 (a,b,c\in \mathbb{R})$ on $G=\Bbb{R}^2\rtimes\Bbb{R}^{+}$.
\begin{prop}
Let $G=\Bbb{R}^2\rtimes\Bbb{R}^{+}$ and $F=\alpha+\beta$ be a Randers metric as above. If $\tilde{F}=e^f F$ is conformally related to $F$, then $\tilde{F}$ is a Douglas metric  if and only if
\begin{align}
&bz\frac{\partial f}{\partial z}-az\frac{\partial f}{\partial x}=b,\label{21}\\
&cz\frac{\partial f}{\partial x}-bz\frac{\partial f}{\partial y}=0,\label{22}\\
&cz\frac{\partial f}{\partial z}-az\frac{\partial f}{\partial y}=c.\label{23}
\end{align}
\end{prop}

\begin{proof}
If $X=ae_1+be_2+ce_3$ and $f$ satisfy the conditions \eqref{21}, \eqref{22} and \eqref{23}, easily it can be seen the relation \eqref{3} holds, and it shows that $\tilde{F}$ is of Douglas type.\\
Conversely, if $\tilde{F}$ is a Douglas metric and $X=ae_1+be_2+ce_3$, based on Proposition \eqref{10}, we have
\begin{equation}
 d\beta(Y,Z)+Yfg(X,Z)-Zfg(X,Y)=0\nonumber
\end{equation}
which is satisfied for all $Y$ and $Z$ in the Lie algebra of $G$.\\
In the special case where $Y=e_i$ and $Z=e_j$ $(i<j, i,j\in\{1,2,3\})$ we have
\begin{align}
&bz\frac{\partial f}{\partial z}-az\frac{\partial f}{\partial x}=b,\label{14}\\
&cz\frac{\partial f}{\partial z}-az\frac{\partial f}{\partial y}=c,\label{15}\\
&cz\frac{\partial f}{\partial x}-bz\frac{\partial f}{\partial y}=0.\label{16}
\end{align}
\end{proof}

\begin{cor}
\begin{description}
  \item[a] If $X=be_2+ce_3$ then, $\tilde{F}$ is of Douglas type if and only if $f=lnz+g(x,y)$ and $cz\frac{\partial f}{\partial x}-bz\frac{\partial f}{\partial y}=0$.
  \item[b] If $X=ae_1+ce_3$ then,$\tilde{F}$ is of Douglas type if and only if $f=ln|z|+c^{'}( c^{'}\in\mathbb{R})$.
  \item[c] If $X=ae_1+be_2$ then,$\tilde{F}$ is of Douglas type if and only if $f=f(x,z)$ and $bz\frac{\partial f}{\partial z}-az\frac{\partial f}{\partial x}=b$.
\end{description}
\end{cor}

\begin{example}
Let $G=\Bbb{R}^2\rtimes\Bbb{R}^{+}$ and $F$ be an $(\alpha,\beta)$-metric defined by the above left-invariant Riemannian metric $g$ and a left-invariant vector field $X$. We see that $\nabla_{e_1}{e_3}=\nabla_{e_2}{e_3}=0$ and $\nabla_{e_3}{e_3}=e_1$. Suppose that $X=ae_3 (a\in\mathbb{R})$, $f(x,y,z)=lnz$ and $Y=y_1e_1+y_2e_2+y_3e_3 (y_1,y_2,y_3\in C^{\infty}(G))$. Easily $\nabla_{Y}{X}=ay_3z\frac{\partial}{\partial z}$, $g(X,Y)=ay_3$, $Xf=0$ and $\nabla f=z\frac{\partial}{\partial z}$. Therefore, according to the Proposition\eqref{9}, $\tilde{F}$ is a Berwald metric.
\end{example}

{\large{\textbf{Acknowledgment.}}} We are grateful to the office of Graduate Studies of the University of Isfahan for their support.

\end{document}